\documentclass[a4paper]{amsart}
\usepackage{lmodern}
\usepackage[T1]{fontenc}
\usepackage{mathtools}
\usepackage{mathrsfs}
\usepackage{amssymb}
\usepackage{tikz-cd}
\usepackage{microtype}

\tikzcdset{arrow style=Latin Modern}

\title{Homological projective duality for linear systems with base locus}
\author{Francesca Carocci}
\author{Zak Tur\v{c}inovi\'c}
\address{Department of Mathematics\\Imperial College London\\London SW7 2AZ\\UK}
\email{f.carocci14@imperial.ac.uk}
\email{z.turcinovic14@imperial.ac.uk}

\newtheorem{thm}{Theorem}[section]
\newtheorem{cor}[thm]{Corollary}
\newtheorem{lem}[thm]{Lemma}
\newtheorem{prop}[thm]{Proposition}

\theoremstyle{remark}
\newtheorem{rem}[thm]{Remark}

\theoremstyle{definition}
\newtheorem{df}[thm]{Definition}

\renewcommand{\P}{\mathbb{P}}
\renewcommand{\O}{\mathcal{O}}
\newcommand{\D}{\mathbf{D}^b}
\newcommand{\C}{\mathcal{C}}
\DeclareMathOperator{\Hom}{Hom}
\DeclarePairedDelimiter{\ang}{\langle}{\rangle}
\newcommand{\F}{\mathscr{F}}
\newcommand{\G}{\mathscr{G}}
\newcommand{\A}{\mathcal{A}}
\newcommand{\B}{\mathcal{B}}
\renewcommand{\H}{\mathcal{H}}
\renewcommand{\TH}{\tilde{\mathcal{H}}}
\newcommand{\TA}{\widetilde{\mathcal{A}}}
\DeclarePairedDelimiterX{\setb}[2]{\{}{\}}{#1\,\delimsize |\,#2}
\DeclareMathOperator{\Ext}{Ext}
\DeclareMathOperator{\Bl}{Bl}
\newcommand{\X}{\tilde{X}}
\newcommand{\E}{\tilde{E}}
\newcommand{\id}{\mathrm{id}}

\newcommand{\CX}{\check{X}}
\newcommand{\cpi}{\check{\pi}}
\newcommand{\CE}{\check{E}}
\newcommand{\cj}{\check{j}}
\newcommand{\cp}{\check{p}}
\newcommand{\ci}{\check{i}}
\renewcommand{\L}{\mathbb{L}}
\newcommand{\R}{\mathbb{R}}
\newcommand{\I}{\mathcal{I}}
\DeclarePairedDelimiter{\abs}{\lvert}{\rvert}
\newcommand{\dash}{\text{--}}
\newcommand{\N}{\mathcal{N}}
\DeclareMathOperator{\Cone}{Cone}
\newcommand{\K}{\mathcal{K}}
\newcommand{\T}{\mathcal{T}}
\DeclareMathOperator{\RHom}{\mathbf{R}Hom}

\newcommand{\TC}{\tilde{\C}}
\newcommand{\TE}{\tilde{E}}
\newcommand{\perf}{\mathrm{perf}}
\newcommand{\tpi}{\tilde{\pi}}
\DeclarePairedDelimiterX{\rest}[2]{.}{\rvert_{#2}}{#1}
\newcommand{\res}{\rest*}
\newcommand{\tj}{\tilde{j}}
\newcommand{\tp}{\tilde{p}}
\newcommand{\ti}{\tilde{i}}

\begin{document}
\begin{abstract}
    We show how blowing up varieties in base loci of linear systems gives a procedure for creating
    new homological projective duals from old.
    
    Starting with a HP dual pair $X,Y$ and smooth orthogonal linear sections~$X_L,Y_L$, we prove
    that the blowup of $X$ in $X_L$ is naturally HP dual to~$Y_L$. The result does not need~$Y$ to
    exist as a variety, i.e.\ it may be "noncommutative".
    
    We extend the result to the case where the base locus~$X_L$ is a multiple of a smooth variety
    and the universal hyperplane has rational singularities; here the HP dual is a categorical
    resolution of singularities of~$Y_L$.

    Finally we give examples where, starting with a noncommutative~$Y$, the above process
    nevertheless gives geometric HP duals.
\end{abstract}
\maketitle
\setcounter{tocdepth}{1}
\tableofcontents

\section{Introduction}
One of the most powerful tools for investigating derived categories of algebraic varieties and their
semiorthogonal decompositions is Kuznetsov's homological projective duality~\cite{kuznetsovhpd}. It
is a beautiful theory but it's hard to produce geometric examples; however some actual examples do
exist~\cite{kuznetsovquadric,kuznetsovgrassmannian,kuznetsovcubic,bernardarabolognesifaenzi,rennemo}.
HP duality starts with a base point free linear system on some variety~$X$; based on a guess of
Calabrese and Thomas~\cite{calabresethomas} we see that for a sublinear system with base locus one
can consider a natural HPD problem on the blowup of~$X$ in the base locus. We show that this new HPD
problem is closely related to the original one. In particular we obtain a procedure for constructing
new HP duals from old:

\begin{thm}
    \label{thm:intro}
    Let $X\to\P(V)$ and~$Y\to\P(V^*)$ be a HP dual pair with respect to the Lefschetz decomposition
    of~$\D(X)$ given by
    \[
        \D(X)=\ang{\A_0,\A_1(1),\dots,\A_{i-1}(i-1)}.
    \]
    Then for any sublinear system~$L\subset V^*$, such that the corresponding linear
    sections~$X_L,Y_L$ have expected dimension, we have an HP dual pair $\Bl_{X_L}X\to\P(L^*)$
    and~$Y_L\to\P(L)$ with respect to the following Lefschetz decomposition of~$\D(\Bl_{X_L}X)$:
    \[
        \D(\Bl_{X_L}X)=\ang*{\TA_0,\TA_1(1),\dots,\TA_{l-2}(l-2)},
    \]
    where~$l=\dim L$ and the pieces~$\TA_k$ are
    \[
        \TA_k\coloneqq\begin{cases}
            \ang*{\pi^*\A_k\otimes\O((l-1)E),\D(X_L)_{-l+1}}&\text{if $k<i$,}\\
            \D(X_L)_{-l+1}&\text{if $k\geq i$.}
        \end{cases}
    \]
    Here~$\pi\colon\Bl_{X_L}X\to X$ is the projection from the blowup and~$E$ is the exceptional
    divisor.
\end{thm}

The fact that this is true can be seen almost immediately by thinking fibrewise. Indeed, by
assumption we know that the interesting part of the derived category of a hyperplane section $X_H$
of~$X\to\P(V)$ is just the derived category of the corresponding fibre of~$Y\to\P(V^*)$:
\[
    \D(X_H)=\ang{\D(Y_H),\A(1),\dots,\A(i-1)}.
\]
Obviously a hyperplane section of~$\Bl_{X_L}X\to\P(L^*)$ is just~$\Bl_{X_L}X_H$ and thus by Orlov's
theorem~\cite{orlov} one notices that the interesting part of its derived category is
also~$\D(Y_H)$. However, we are restricting our attention to only those hyperplane sections that
contain~$X_L$ and thus the expected HP dual must be the restriction~$Y_L\to\P(L)$ of the original
one.

The proof now consists of making the above work in families by writing the universal hyperplane
section~$\TH$ of the blowup as a blowup itself:
\begin{equation}
    \label{eq:intro}
    \TH=\Bl_{X_L\times\P(L)}\H_L.
\end{equation}
Afterwards it's just a matter of applying Orlov's theorem and performing some mutations to obtain
the result. The whole story works actually more generally if we start with a noncommutative HP dual.
Furthermore, if we start with a linear system~$L$ whose base locus is of the form~$mZ$ for some
smooth variety $Z$ with~$m\geq 1$, then we have the following:

\begin{thm}
    There is a natural Lefschetz decomposition for~$\D(\Bl_ZX)$ with respect to the line
    bundle~$\pi^*\O_X(1)(-mE)$ such that, if~$\H_L$ has only rational singularities, the HP dual
    of~$\Bl_ZX\to\P(L)$ is a categorical resolution of singularities~\cite{kuznetsovsingularities}
    of the interesting part of~$\D(\H_L)$.
\end{thm}

The idea and the techniques of the proof are basically the same as in Theorem~\ref{thm:intro}. We
still have the isomorphism~\eqref{eq:intro} which tells us that~$\D(\TH)$ is a categorical
resolution of singularities of~$\D(\H_L)$. Restricting to the relevant subcategories the result
immediately follows.

Allowing base locus with multiplicity leads us to consider the final two examples where we see an
interesting phenomenon: starting with a noncommutative HP dual the blowing up process yields a
geometric HP dual pair. More precisely we consider the degree~$3$ Veronese embedding
\[
    \P^5\hookrightarrow\P(H^0(\O_{\P^5}(3))^*)
\]
with the standard Lefschetz decomposition
of~$\D(\P^5)$ with respect to~$\O_{\P^5}(3)$. It is known that the HP dual in this case is a
noncommutative K3-fibration
\[
    \C\to\P(H^0(\O_{\P^5}(3))).
\]
By setting~$L$ to be the linear system of cubics that are singular at a fixed point~$P\in\P^5$ we
obtain a base point free linear system $L=H^0(\pi^*\O_{\P^5}(3)(-2E))$ on $\Bl_P\P^5$,
where~$\pi\colon\Bl_P\P^5\to\P^5$ is the blowup and~$E$ is its exceptional divisor. We can
equip~$\D(\Bl_P\P^5)$ with a Lefschetz decomposition (similar to the one in Theorem~\ref{thm:intro})
with respect to the line bundle~$\pi^*\O_{\P^5}(3)(-2E)$. We show that the HP dual
of~$\Bl_P\P^5\to\P(L^*)$ is generically a K3-fibration $\CX\to\P(L)$,
where~$\CX\subset\P^4\times\P(L)$ is the intersection of a universal~$(2,1)$ and a universal~$(3,1)$
divisor.

For the second example we take~$L$ to be the linear system of cubics containing a fixed
plane~$\P(W)\cong\P^2\subset\P^5$. The HP dual we obtain in this case is the noncommutative variety
$(\P(W')\times\P(L),\C_0)$, where~$\P(W')$ is the orthogonal complement of $\P(W)\subset\P^5$
and~$\C_0$ is an even Clifford algebra sheaf on~$\P(W')\times\P(L)$. Both of these results are
straightforward extensions of Kuznetsov's results on cubic fourfolds~\cite{kuznetsovcubic}.

\subsection*{Acknowledgements}
We are pleased to thank Richard Thomas for posing the problem, for teaching us the subject and for
all the helpful discussions.  Further thanks go to the members of the first year LSGNT cohort for
all the helpful conversations; special thanks go to Luca Battistella, Pierrick Bousseau, Otto
Overkamp and Jakub Witaszek.

\section{Preliminaries}

\subsection{Notation and conventions}
Let~$X$ be an algebraic variety. We always work over an algebraically closed field~$\mathbb{K}$ of
characteristic zero. We will denote the bounded derived category of coherent sheaves on $X$
by~$\D(X)$. We will abuse notation and denote the total derived tensor product by~$\otimes$ and the
total derived pushforward and pullback of a map $f$ by $f_*$ and ~$f^*$, respectively. For a sheaf
$\F$ on~$X$ we will denote its dual by~$\F^{\vee}$, whereas for a vector space~$V$ we will denote
its dual by~$V^*$. For two objects~$\F,\G\in\D(X)$ we will denote by~$\Hom(\F,\G)$ the set
of morphisms from $\F$ to $\G$ in $\D(X)$, whereas~$\RHom(\F,\G)$ shall denote the derived
global~$\Hom$. In particular we identify
\[
    \Hom(\F,\G[k])=H^0(\RHom(\F,\G[k]))=\Ext^k(\F,\G).
\]

\subsection{Homological projective duality}
\label{sec:hpd}
Consider a smooth projective variety~$X$ together with a regular map~$f\colon X\to\P(V)$ for some
finite-dimensional vector space~$V$. Without loss of generality we will assume that the image of~$f$
is not contained in a hyperplane. Note that this is equivalent to giving an effective line bundle
$\O_X(1)\coloneqq f^*\O_{\P(V)}(1)$ on~$X$ together with a base point free linear
system~$V^*\subseteq H^0(\O_X(1))$.

\begin{df}
    A \emph{Lefschetz decomposition} of~$\D(X)$ with respect to a fixed line bundle~$\O_X(1)$ is a
    semiorthogonal decomposition of the form
    \begin{equation}
        \label{eq:stdlef}
        \D(X)=\ang{\A_0,\A_1(1),\dots,\A_{i-1}(i-1)},
    \end{equation}
    where~$0\subset\A_{i-1}\subset\dots\subset\A_0\subset\D(X)$ are full triangulated subcategories.
    Here,
    \[
        \A_k(k)=\setb{\F\otimes\O_X(k)}{\F\in\A_k}.
    \]
    If~$\A_0=\dots=\A_{i-1}$ we call the Lefschetz decomposition \emph{rectangular}.
\end{df}

\begin{rem}
    \label{prop:hyperplane}
    Recall that a Lefschetz decomposition gets its name from the fact that for any hyperplane
    section $X_H$ of~$X\to\P(V)$ we have:
    \begin{enumerate}
        \item The restriction of the derived pullback
            functor~$i^*|_{\A_k(k)}\colon\A_k(k)\to\D(X_H)$ is fully faithful for~$k\geq 1$,
        \item $i^*(\A_1(1)),\dots,i^*(\A_{i-1}(i-1))$ form a semiorthogonal collection in~$\D(X_H)$.
    \end{enumerate}
    Consequently we have a semiorthogonal decomposition of the form
    \[
        \D(X_H)=\ang{\C_H,\A_1(1),\dots,\A_{i-1}(i-1)},
    \]
    for some triangulated subcategory~$\C_H$.
\end{rem}

Let~$\H$ be the universal hyperplane section of~$X\to\P(V)$. Explicitly this means
\[
    \H=\setb{(p,[s])\in X\times\P(V^*)}{s(p)=0},
\]
which is the zero locus of the tautological section in
\[
    H^0(\O_{X\times\P(V^*)}(1,1))\cong H^0(\O_X(1))\otimes V.
\]
We now recall Kuznetsov's definition of \emph{homological projective duals}:

\begin{df}
    \label{df:hpd}
    Let~$X\to\P(V)$ be equipped with a Lefschetz decomposition as above. A projective variety~$Y$
    together with a regular map~$Y\to\P(V^*)$ is called \emph{homological projective dual} of~$X$ if
    there is an object~$\F\in\D(Y\times_{\P(V^*)}\H)$ such that the corresponding Fourier--Mukai
    transform~$\Phi_{Y\to\H}^{\F}\colon\D(Y)\to\D(\H)$ is fully faithful and there is a
    semiorthogonal decomposition
    \[
        \D(\H)=\ang*{\Phi_{Y\to\H}^{\F}\left(\D(Y)\right),\A_1(1)\boxtimes\D(\P(V^*)),\dots,\A_{i-1}(i-1)\boxtimes\D(\P(V^*))}.
    \]
    Here~$\A_k(k)\boxtimes\D(\P(V^*))$ is shorthand
    for~$i_{\H}^*\left(\A_k(k)\boxtimes\D(\P(V^*))\right)$, where~$i_{\H}$ is the inclusion of the
    divisor~$\H\subset X\times\P(V^*)$.
\end{df}

Let~$L\subset V^*$ be a sublinear system and let~$l$ denote its dimension. In the following we
denote its base locus by~$X_L\coloneqq X\times_{\P(V)}\P(L^{\perp})$. For a homological projective
dual~$Y\to\P(V^*)$ we will write~$Y_L\coloneqq Y\times_{\P(V^*)}\P(L)$. Kuznetsov's main theorem on
homological projective duality~\cite{kuznetsovhpd} now tells us precisely what
Definition~\ref{df:hpd} does for us in terms of the linear sections $X_L$ of~$X$:

\begin{thm}[Kuznetsov]
    \label{thm:hpdmain}
    Let~$X\to\P(V)$ and its Lefschetz decomposition be as above and assume~$Y\to\P(V^*)$ is a
    homologically projectively dual variety. Assume furthermore that $\dim V>i$, where~$i$ is the
    number of terms in the Lefschetz decomposition of~$\D(X)$. Then we have the following facts:
    \begin{enumerate}
        \item $Y$ is smooth and~$\D(Y)$ has a dual Lefschetz decomposition given by
            \[
                \D(Y)=\ang{\B_{j-1}(1-j),\dots,\B_1(-1),\B_0},
            \]
            with~$0\subset\B_{j-1}\subset\dots\subset\B_0\subset\D(Y)$ and~$j=\dim
            V-\max\setb{k}{\A_k=\A_0}$. Here the~$\B_k$ are twisted by the pullback
            of~$\O_{\P(V^*)}(1)$.
        \item If $\dim X_L=\dim X-l$ and~$\dim Y_L=\dim Y-\dim V+l$ then there is a triangulated
            category~$\C_L$ such that we have semiorthogonal decompositions
            \begin{align*}
                \D(X_L)&=\ang{\C_L,\A_l(1),\dots,\A_{i-1}(i-l)}\\
                \D(Y_L)&=\ang{\B_{j-1}(\dim V-l-j),\dots,\B_{\dim V-l}(-1),\C_L}.
            \end{align*}
    \end{enumerate}
\end{thm}

\begin{rem}
    Note that a \emph{categorical}\footnote{We will often use "categorical" and "noncommutative"
        interchangeably.} HP dual~$\C$ always exists tautologically. Namely we can just define it to
    be the right orthogonal of the trivial part of~$\D(\H)$:
    \[
        \C\coloneqq\ang*{\A_1(1)\boxtimes\D(\P(V^*)),\dots,\A_{i-1}(i-1)\boxtimes\D(\P(V^*))}^{\perp}.
    \]
    One can then still develop much of the general theory in this categorical setting; we refer to
    Thomas' notes~\cite{thomashpd} for a good introduction to that point of view.
\end{rem}

\section{Linear systems with base locus}
Assume we are given~$f\colon X\to\P(V)$ with some Lefschetz decomposition and a corresponding
homological projective dual~$Y\to\P(V^*)$. Also assume we have a linear subspace~$L\subset V^*$ of
dimension~$l$ such that~$X_L$ is smooth, $\dim X_L=\dim X-l$, and~$\dim Y_L=\dim Y-\dim V+l$. Note
that there is a natural map~$Y_L\to\P(L)$ and a rational map~$X\dashrightarrow\P(L^*)$.  We consider
the blow up of the rational map in its indeterminacy locus $X_L$ to obtain a regular map:
\begin{equation}
    \label{eq:diag0}
    \begin{tikzcd}
        E\ar[r,hook,"j"]\ar[d,"p"]&\Bl_{X_L}X\ar[d,"\pi"]\ar[dr,"\phi"]&\\
        X_L\ar[r,hook,"i"]&X\ar[r,dashed]&\P(L^*)\rlap{\ ,}
    \end{tikzcd}
\end{equation}
where by construction we have
\[
    \phi^*\O_{\P(L^*)}(1)\cong\pi^*\O_X(1)(-E),
\]
and~$E$ is the exceptional divisor of the blowup. From now on write~$\X\coloneqq\Bl_{X_L}X$ and we
obviously always assume~$l\geq 2$.  In order to show that $\X\to\P(L^*)$ and~$Y_L\to\P(L)$ are
HP dual we have to fix a Lefschetz decomposition for~$\D(\X)$. For this we will use Orlov's
decomposition for a blowup~\cite{orlov}.

\subsection{Stupid Lefschetz decomposition}
We first look at the easiest case: when~$\D(X)$ is endowed with the stupid Lefschetz decomposition
and the homological projective dual is just the universal hyperplane section~$\H\to\P(V^*)$. In this
case Orlov's theorem then gives us a semiorthogonal decomposition of the form
\[
    \D(\X)=\ang*{\pi^*\D(X)\otimes\omega_{\X},\D(X_L)_{-l+1},\dots,\D(X_L)_{-1}}.
\]
If we now set
\begin{equation}
    \label{eq:lefstupid}
    \TA_0\coloneqq\ang*{\pi^*\D(X)\otimes\omega_{\X},\D(X_L)_{-l+1}},\quad\TA_1\coloneqq\dots\coloneqq\TA_{l-2}\coloneqq\D(X_L)_{-l+1},
\end{equation}
then we have a Lefschetz decomposition of~$\D(\X)$:

\begin{prop}
    \label{prop:lefstupid}
    The above is a Lefschetz decomposition of~$\D(\X)$ with respect to the line
    bundle~$\pi^*\O_X(1)(-E)$, i.e.\ we have a semiorthogonal decomposition of the form
    \[
        \D(\X)=\ang*{\TA_0,\TA_1(1),\dots,\TA_{l-2}(l-2)},
    \]
    with~$0\subset\TA_{l-2}\subset\dots\subset\TA_0\subset\D(\X)$.
\end{prop}

\begin{proof}
    All that's left to show is that~$\D(X_L)_k\otimes\pi^*\O_X(1)(-E)=\D(X_L)_{k+1}$, for any~$k$.
    But this follows immediately from~$j^*\O(-E)\cong\O_E(1)$ and the projection formula.
\end{proof}

\begin{prop}
    \label{thm:lefstupid}
    Let~$\X\to\P(L^*)$ be endowed with the Lefschetz decomposition~\eqref{eq:lefstupid}. Then its
    homological projective dual is~$\H_L\to\P(L)$.
\end{prop}

\begin{proof}
    Let~$\TH$ denote the universal hyperplane section of~$\X$ with respect to the line
    bundle~$\pi^*\O_X(1)(-E)$. We need to show that there is an
    object~$\F\in\D(\H_L\times_{\P(L)}\TH)$ such that its associated Fourier--Mukai
    functor~$\Phi_{\H_L\to\TH}^{\F}\colon\D(\H_L)\to\D(\TH)$ is fully faithful and we have a
    semiorthogonal decomposition
    \[
        \D(\TH)=\ang*{\Phi_{\H_L\to\TH}^{\F}\left(\D(\H_L)\right),\D(X_L)_{-l+2}\boxtimes\D(\P(L)),\dots,\D(X_L)_{-1}\boxtimes\D(\P(L))}.
    \]
    The easiest way to do that is by observing that $\TH\cong\Bl_{X_L\times\P(L)}\H_L$. Thinking
    fibrewise this is obvious: hyperplane sections of $\phi$ are just strict transforms of
    hyperplane sections of $f\colon X\to\P(V)$ containing the base locus $X_L$. In families, just
    look at $\H_L\in X\times\P(V)$, and repeat the argument to say that $\TH$ is the strict
    transform in $\pi\times\id_{\P(L)}\colon\Bl_{X_L}X\times\P(L)\to X\times\P(L)$ of $\H_L$.  We
    will denote the projection by~$\tilde{\pi}\colon\TH\to\H_L$ and we have the commutative diagram
    \begin{equation}
        \label{eq:diag1}
        \begin{tikzcd}
            &\E\ar[dl,hook]\ar[dd,"\tilde{p}" near
            end]\ar[rr,"\tilde{j}",hook]&&\TH\ar[dl,hook,"i_{\TH}"]\ar[dd,"\tilde{\pi}"]\\
            E\times\P(L)\ar[dd,"p\times\id_{\P(L)}"]\ar[rr,crossing over,"j\times\id_{\P(L)}" near
            end,hook]&&\X\times\P(L)&\\ &X_L\times\P(L)\ar[dl,equals]\ar[rr,"\tilde{i}" near
            end,hook]&&\H_L\ar[dl,"i_{\H_L}",hook]\\
            X_L\times\P(L)\ar[rr,hook,"i\times\id_{\P(L)}"]&&X\times\P(L)\ar[uu,"\pi\times\id_{\P(L)}"
            near start,crossing over,leftarrow]&\rlap{\ ,}
        \end{tikzcd}
    \end{equation}
    where the front and back faces are blowup diagrams. Since the codimension of $X_L\times\P(L)$ in
    $\H_L$ is~$l-1$, applying Orlov's theorem to~$\TH$ we obtain the semiorthogonal decomposition
    \[
        \D(\TH)=\ang*{\tilde{\pi}^*\D(\H_L)\otimes\omega_{\TH},\D(X_L\times\P(L))_{-l+2},\dots,\D(X_L\times\P(L))_{-1}}.
    \]
    Note that we have~$\O(\E)\cong i_{\TH}^*(\O(E)\boxtimes\O_{\P(L)})$ and the top face of the
    diagram~\eqref{eq:diag1} is an exact Cartesian square, and thus one can easily compute
    \[
        \D(X_L\times\P(L))_k=i_{\TH}^*(\D(X_L)_k\boxtimes\D(\P(L))).
    \]
    Finally note that~$\tilde{\pi}(\dash)\otimes\omega_{\TH}$ can obviously be written as a
    Fourier--Mukai functor with kernel pushed forward from the fiber product $\TH\times_{\P(L)}
    \H_L$. Indeed the kernel is just given by the pushforward of the canonical bundle along the
    graph map:
    \[
        (\tilde{\pi}\times\id_{\TH})_*\omega_{\TH}.
    \]
    As the graph map is linear over $\P(L)$ the pushforward factors via~$\TH\times_{\P(L)}\H_L$.
\end{proof}

\begin{rem}
    Note that to use Orlov's decomposition in the proof we implicitly used the fact that~$\H_L$ is
    smooth whenever~$X_L$ is. Indeed, notice that~$\H_L\to X$ is a smooth $\P^{l-2}$-bundle away
    from $X_L$, and~$\H_L\to\P(L)$ is a smooth bundle near~$X_L$: every element of the linear
    system~$L$ is smooth near the base locus~$X_L$.
\end{rem}

\subsection{General Lefschetz decomposition}
Assume now that~$\D(X)$ is endowed with an arbitrary Lefschetz decomposition of the form
\begin{equation}
    \label{eq:stdlef2}
    \D(X)=\ang{\A_0,\A_1(1),\dots,\A_{i-1}(i-1)},
\end{equation}
and that~$Y\to\P(V^*)$ is a corresponding homological projective dual. Again we are hoping
that~$\X\to\P(L^*)$ equipped with some Lefschetz decomposition is homologically projectively dual
to~$Y_L\to\P(L)$. If we want to construct a Lefschetz decomposition of~$\D(\X)$ using all the
information from~\eqref{eq:stdlef2}, one reasonable way of doing that is by interlacing the pieces
from~\eqref{eq:stdlef2} and the pieces~$\D(X_L)_k$ coming from Orlov's theorem. Note that this way
we expect to get a Lefschetz decomposition with~$\max\{i,l-1\}$ terms.  Thus to make
Theorem~\ref{thm:hpdmain} apply we will assume from now on~$i\leq l-1$.

In order to interlace the pieces we will need to perform a series of mutations. The necessary
mutations can be described generally in any semiorthogonal decomposition coming from Orlov's
theorem:

\begin{prop}
    \label{prop:lefgeneral}
    Consider the blowup of a smooth variety~$X$ in a smooth subvariety~$Z$ and let the notation be
    analogous to the one in the diagram~\eqref{eq:diag0}. Then for any~$\F\in\D(X)$ and any
    integer~$k$ we have the following isomorphism:
    \[
        \L_{\D(Z)_{-k}}(\pi^*\F\otimes\O((k-1)E))\cong\pi^*\F\otimes\O(kE).
    \]
\end{prop}

\begin{proof}
    We will write~$\Phi_k$ for the embedding of $\D(Z)$ into~$\D(\Bl_ZX)$, i.e.
    \[
        \Phi_k(\dash)\coloneqq j_*p^*(\dash)\otimes\O(kE).
    \]
    Now recall that left mutation through~$\Phi_k(\D(Z))$ is by definition given by the
    distinguished triangle
    \[
        \Phi_k(\Phi_k^!(\F))\to\F\to\L_{\Phi_k(\D(Z))}(\F)\to\Phi_k(\Phi_k^!(\F))[1],
    \]
    for any~$\F\in\D(\Bl_ZX)$, where~$\Phi_k^!$ denotes the right adjoint of~$\Phi_k$. Using the
    fact that~$j$ is an embedding of a divisor we obtain
    \[
        \Phi_k^!(\dash)=p_*j^*(\dash\otimes\O((1-k)E))[-1].
    \]
    Thus for any~$\F\in\D(X)$ we compute using~$p_*\O_E\cong\O_Z$:
    \[
        \Phi_k^!(\pi^*\F((k-1)E))=p_*j^*\pi^*\F[-1]=p_*p^*i^*\F[-1]\cong i^*\F[-1],
    \]
    and in particular we have
    \[
        \Phi_k(\Phi_k^!(\pi^*\F\otimes\O((k-1)E)))\cong j_*p^*i^*\F(kE)[-1]\cong j_*j^*\pi^*\F(kE)[-1].
    \]
    Derived tensoring the short exact sequence
    \[
        0\to\O((k-1)E)\to\O(kE)\to j_*j^*\O(kE)\to 0
    \]
    with~$\pi^*\F$ and shifting the resulting distinguished triangle we obtain
    \[
        \L_{\Phi_k(\D(Z))}(\pi^*\F\otimes\O((k-1)E))\cong\pi^*\F\otimes\O(kE).\qedhere
    \]
\end{proof}

Thus we obtain
\begin{cor}
    If we define
\begin{equation}
    \label{eq:lefgeneral}
    \TA_k\coloneqq\begin{cases}
        \ang*{\pi^*\A_k\otimes\O((l-1)E),\D(X_L)_{-l+1}}&\text{if $k<i$,}\\
        \D(X_L)_{-l+1}&\text{if $k\geq i$,}
    \end{cases}
\end{equation}
for~$0\leq k\leq l-2$, then we get a Lefschetz decomposition of~$\D(\X)$ with respect
to~$\pi^*\O_X(1)(-E)$.
\end{cor}

Note that the Lefschetz decomposition~\eqref{eq:lefgeneral} indeed specialises
to~\eqref{eq:lefstupid} in the case of the stupid Lefschetz decomposition on~$\D(X)$.

\begin{thm}
    \label{thm:lefgeneral}
    Let~$\X\to\P(L^*)$ be endowed with the Lefschetz decomposition~\eqref{eq:lefgeneral}. Then its
    homological projective dual is~$Y_L\to\P(L)$.
\end{thm}

\begin{proof}
    Unless noted otherwise we use the same notation as in Proposition~\ref{thm:lefstupid}. Thus we
    have to find an object~$\F\in\D(Y_L\times_{\P(L)}\TH)$ such that its associated Fourier--Mukai
    transform~$\Phi_{Y_L\to\TH}^{\F}$ is fully faithful and we have a semiorthogonal decomposition
    \begin{equation}
        \label{eq:generalhyperplane}
        \D(\TH)=\ang[\Big]{\Phi_{Y_L\to\TH}^{\F}\left(\D(Y_L)\right),\TA_1(1)\boxtimes\D(\P(L)),\dots,\TA_{l-2}(l-2)\boxtimes\D(\P(L))},
    \end{equation}
    where now the~$\TA_k$ are defined as in~\eqref{eq:lefgeneral}. Since we assumed~$X\to\P(V)$ to
    be homologically projectively dual to~$Y\to\P(V^*)$, there is an
    object~$\G\in\D(Y\times_{\P(V^*)}\H)$ and a semiorthogonal decomposition
    \[
        \D(\H)=\ang*{\Phi_{Y\to\H}^{\G}\left(\D(Y)\right),\A_1(1)\boxtimes\D(\P(V^*)),\dots,\A_{i-1}(i-1)\boxtimes\D(\P(V^*))}.
    \]
    Since we assumed the dimensions of $X_L$ and~$Y_L$ to be as expected we can apply \emph{faithful
        base change}~\cite{kuznetsovhpd} to obtain an object~$\G_L\in\D(Y_L\times_{\P(L)}\H_L)$ and
    a semiorthogonal decomposition
    \[
        \D(\H_L)=\ang*{\Phi_{Y_L\to\H_L}^{\G_L}\left(\D(Y_L)\right),\A_1(1)\boxtimes\D(\P(L)),\dots,\A_{i-1}(i-1)\boxtimes\D(\P(L))}.
    \]
    It follows that we have the following semiorthogonal decomposition for $\TH$:
    \begin{multline}
        \label{eq:orlohyperplane}
        \D(\TH)=\ang[\Big]{\D(X_L\times\P(L))_{-l+2},\dots,\D(X_L\times\P(L))_{-1},\\
            \tilde{\pi}^*\left(\Phi_{Y_L\to\H_L}^{\G_L}(\D(Y_L))\right),\pi^*\A_1(1)\boxtimes\D(\P(L)),\dots,\pi^*\A_{i-1}(i-1)\boxtimes\D(\P(L))}.
    \end{multline}
    By Proposition~\ref{prop:lefgeneral}, left mutation through pieces of the form
    $\D(X_L\times\P(L))_{-l+k}$ just means tensoring by
    \[
        \O(\tilde{E})=i_{\TH}^*\O(E,0).
    \]
    It follows exactly as in Proposition~\ref{prop:lefgeneral} that one can mutate
    \eqref{eq:orlohyperplane} into \eqref{eq:generalhyperplane}. Finally, notice that $\D(Y_L)$ is
    embedded in $\D(\TH)$ via a Fourier--Mukai transform whose kernel is pushed forward from the
    fibre product $\TH\times_{\P(L)}Y_L$. Indeed
    \[
        \tilde{\pi}^*\left(\Phi_{Y_L\to\H_L}^{\G_L}(\dash)\right)\otimes\O((l-2)\tilde{E})
    \]
    is given by $\Phi_{Y_L\to\TH}^{\mathcal E_L}$ where $\mathcal E_L\in \D(Y_L\times \TH)$ is
    \[
        \mathcal{E}_L=(\id_{Y_L}\times\tilde{\pi})^*\G_L\otimes p_{\TH}^*\O((l-2)\tilde{E}).
    \]
    As $\tilde{\pi}$ is $\P(L)$-linear and by hypothesis $\G_L$ is pushed forward from the fiber
    product, one sees that $\mathcal E_L$ is as well pushed forward from $Y_L\times_{\P(L)}\TH$.
\end{proof}

\subsection{Base locus with multiplicity}
Recall that there is a purely categorical notion of HP dual. From that point of view then we can
define the categorical analogue $\C_L$ of~$Y_L$ to just be the right orthogonal of the trivial part
of~$\H_L$:
\[
    \C_L\coloneqq\ang*{\A_1(1)\boxtimes\D(\P(L)),\dots,\A_{i-1}(i-1)\boxtimes\D(\P(L))}^{\perp}.
\]
Then the proof of Theorem~\ref{thm:lefgeneral} in particular also shows that the categorical HP dual
$\TC$ of~$\Bl_{X_L}X\to\P(L^*)$ is just~$\C_L$. Note that we are not assuming the existence of a
geometric HP dual~$Y$ here. Thus, considering that a categorical HP dual always exists, this story
works any time~$X_L$ is smooth and it has the expected codimension. In fact, if we are happy with
the purely categorical result we can additionally drop the assumptions on smoothness and correct
dimension of~$X_L$. In this case the category~$\C_L$ won't be "smooth" and so the correct HP dual
will turn out to be a categorical resolution of it.

Let~$Z\subset X$ be a smooth subvariety and consider the sublinear system~$L$ of all sections
of~$\O_X(1)$ vanishing with order at least $m\geq 1$ along~$Z$:
\[
    L=H^0(\I_Z^{\otimes m}(1))\subset H^0(\O_X(1)).
\]
If we consider the blowup~$\Bl_ZX$ and let the notation be as in the diagram~\eqref{eq:diag0}, then
we can write~$L$ as
\[
    L=H^0(\pi^*\O_X(1)(-mE)).
\]
Note that $Z$ is just~$X_L$ with the reduced scheme structure and thus we again obtain a regular
map~$\Bl_ZX\to\P(L^*)$. Also note that for~$m>1$ the restricted universal hyperplane section~$\H_L$
is not smooth and thus we are not in the hypothesis of Orlov's theorem anymore. However, if we
assume that the singularities of~$\H_L$ are nice enough we still have a HPD story:

\begin{prop}
    \label{prop:lefmult}
    Let~$Z\subset X$ be a smooth subvariety of codimension $c$ and $m\geq 1$; set
    $a\coloneqq\lceil\frac{c-1}{m}\rceil$ and assume~$i\leq a$. Then there is a Lefschetz
    decomposition of~$\D(\Bl_ZX)$ of the form
    \[
        \D(\Bl_ZX)=\ang*{\TA_0,\TA_1(1),\dots,\TA_{a-1}(a-1)},
    \]
    with respect to the line bundle~$\pi^*\O_X(1)(-mE)$, where the~$\TA_k$ are defined as
    \[
        \TA_k\coloneqq\begin{cases}
            \ang*{\pi^*\A_k\otimes\O((c-1)E),\D(Z)_{-c+1},\dots,\D(Z)_{-c+m}}&\text{if $0\leq k\leq
                i-1$,}\\
            \ang*{\D(Z)_{-c+1},\dots,\D(Z)_{-c+m}}&\text{if $i\leq k\leq a-2$,}\\
            \ang*{\D(Z)_{-c+1},\dots,\D(Z)_{-c+((c-1)\bmod m)}}&\text{if $k=a-1$.}
        \end{cases}
    \]
\end{prop}

\begin{proof}
    This is an immediate consequence Proposition~\ref{prop:lefgeneral}.
\end{proof}

\begin{thm}
    \label{thm:lefmult}
    Let~$\Bl_ZX\to\P(L^*)$ be as above and equip it with the Lefschetz decomposition of
    Proposition~\ref{prop:lefmult}. Assume furthermore that~$\H_L$ has only rational singularities.
    Then the categorical HP dual~$\TC$ is a categorical resolution of singularities of~$\C_L$.
\end{thm}

\begin{proof}
    Recall that we have the two semiorthogonal decompositions
    \begin{align}
        \label{eq:lefmult}
        \D(\TH)&=\ang*{\TC,\TA_1(1)\boxtimes\D(\P(L)),\dots,\TA_{a-1}(a-1)\boxtimes\D(\P(L))}\\
        \D(\H_L)&=\ang*{\C_L,\A_1(1)\boxtimes\D(\P(L)),\dots,\A_{i-1}(i-1)\boxtimes\D(\P(L))}.
    \end{align}
    We apply Proposition~\ref{prop:lefgeneral} to mutate~\eqref{eq:lefmult} into the decomposition
    \begin{multline*}
        \D(\TH)=\ang{\D(Z)_{-c+m+1}\boxtimes\D(\P(L)),\dots,\D(Z)_{-1}\boxtimes\D(\P(L)),\TC'\\
            \pi^*\A_1(1)\boxtimes\D(\P(L)),\dots,\pi^*\A_{i-1}(i-1)\boxtimes\D(\P(L))},
    \end{multline*}
    where~$\TC'\cong\TC$. We now show that~$\TC'$ is a categorical resolution of
    singularities~\cite{kuznetsovsingularities} of~$\C_L$, i.e.\ there exists a pair of functors
    \[
        \sigma_*\colon\TC'\to\C_L,\quad\sigma^*\colon\C_L^{\perf}\to\TC',
    \]
    such that~$\sigma^*$ is left adjoint to~$\sigma_*$ and the
    unit~$\id_{\C_L^{\perf}}\to\sigma_*\sigma^*$ is an isomorphism. On the level of the ambient
    categories we have such a pair given by $\tpi_*$ and~$\tpi^*$. We claim that the restrictions
    \[
        \sigma_*\coloneqq\res{\tpi_*}{\TC'},\quad\sigma^*\coloneqq\res{\tpi^*}{\C_L^{\perf}},
    \]
    do the job. Indeed, since we assumed that~$\H_L$ has only rational singularities we
    have~$\sigma_*\sigma^*\cong\id_{\C_L^{\perf}}$, and thus we only need to show that $\sigma_*$
    and~$\sigma^*$ have the right codomains. For~$\sigma_*$ this follows from the computation
    \[
        \Hom(i_{\H_L}^*(\A_k(k)\boxtimes\D(\P(L))),\tpi_*\TC')\cong\Hom(i_{\TH}^*(\pi^*\A_k(k)\boxtimes\D(\P(L))),\TC')=0.
    \]
    For~$\sigma^*$ it follows from
    \[
        \Hom(i_{\TH}^*(\pi^*\A_k(k)\boxtimes\D(\P(L))),\tpi^*\C_L^{\perf})\cong\Hom(i_{\H_L}^*(\A_k(k)\boxtimes\D(\P(L))),\C_L^{\perf})=0,
    \]
    and
    \begin{align*}
        \Hom(\tpi^*\C_L^{\perf},i_{\TH}^*(\D(Z)_k\boxtimes\D(\P(L))))&\cong\Hom(\tpi^*\C_L^{\perf},\tj_*\tp^*\D(Z\times\P(L))\otimes\O(-k\TE))\\
        &\cong\Hom(\tp^*\ti^*\C_L^{\perf},\tp^*\D(Z\times\P(L))\otimes\O_{\TE}(k))\\
        &\cong\Hom(\ti^*\C_L^{\perf},\D(Z\times\P(L))\otimes\tp_*\O_{\TE}(k))\\
        &=0,
    \end{align*}
    for~$-c+1\leq k\leq -1$.
\end{proof}

\section{Two examples}
\label{sec:examples}
A natural question to ask at this point is whether in the case of base locus with multiplicity we can
say something more about the categorical HP dual~$\TC$ instead of just that it is a categorical
resolution of singularities of~$\C_L$. For example: could~$\TC$ be geometric? We will present two
examples showing that this is sometimes the case, depending on the choice of linear system~$L$. More
precisely the starting point will be the degree~$3$ Veronese
embedding~$\P^5\hookrightarrow\P(H^0(\O_{\P^5}(3))^*)$ equipped with the standard Lefschetz
decomposition of~$\P^5$ with respect to~$\O_{\P^5}(3)$. The HP dual in this case is a noncommutative
K3-fibration. Both examples are essentially applications of Kuznetsov's results on cubic
fourfolds~\cite{kuznetsovcubic}.

\subsection{First example}
We start with an example where the above process yields geometric HP duality. We will take~$L\subset
H^0(\O_{\P^5}(3))$ to be the linear system of all cubic fourfolds that are singular at a fixed
point~$P\in\P^5$. In particular then on~$\Bl_P\P^5$ we have
\[
    L\cong H^0(\pi^*\O_{\P^5}(3)(-2E)),
\]
and it is base point free. Finally by Proposition~\ref{prop:lefmult} we have a rectangular Lefschetz
decomposition of~$\Bl_P\P^5$ with respect to~$\pi^*\O_{\P^5}(3)(-2E)$ given by
\begin{equation}
    \label{eq:lef1stexample}
    \D(\Bl_P\P^5)=\ang*{\TA,\TA(1)},
\end{equation}
where
\[
    \TA=\ang{\pi^*\O_{\P^5}(-6)(4E),\pi^*\O_{\P^5}(-5)(4E),\pi^*\O_{\P^5}(-4)(4E),\D(P)_{-4},\D(P)_{-3}}.
\]

\begin{prop}
    The HP dual of~$\Bl_P\P^5\to\P(L^*)$ with respect to the above Lefschetz decomposition is a
    complete intersection~$\CX$ of a universal $(2,1)$ and a universal~$(3,1)$ divisor
    in~$\P^4\times\P(L)$. In particular~$\CX\to\P(L)$ is generically a K3-fibration.
\end{prop}

To prove this proposition we follow closely Calabrese and Thomas~\cite{calabresethomas}.  Note first
that linear projection away from~$P$ defines a rational map~$\P^5\dashrightarrow\P^4$ with
indeterminacy locus~$P$ and thus induces a regular map~$\phi\colon\Bl_P\P^5\to\P^4$ which
exhibits~$\Bl_P\P^5$ as a $\P^1$-bundle over~$\P^4$ given by
\[
    \phi\colon\P(\O_{\P^4}(1)\oplus\O_{\P^4})\to\P^4.
\]
This carries a tautological line bundle which we denote by~$\O_{\phi}(-1)$. From this point of view
the exceptional divisor~$E$ is cut out by a section of~$\O_{\phi}(1)$ which is the image of the
section~$(0,1)\in H^0(\O_{\P^4}(-1)\oplus\O_{\P^4})$ under the tautological isomorphism
\[
    H^0(\O_{\phi}(1))\cong H^0(\O_{\P^4}(-1)\oplus\O_{\P^4}),
\]
and thus we obtain
\[
    \O_{\phi}(1)\cong\O(E).
\]
Using the fact that~$\phi^*\O_{\P^4}(1)\cong\pi^*\O_{\P^5}(1)(-E)$ and the projection formula we can
now compute
\[
    \phi_*\pi^*\O_{\P^5}(3)(-2E)\cong\O_{\P^4}(2)\oplus\O_{\P^4}(3).
\]
In particular we have
\[
    H^0(\pi^*\O_{\P^5}(3)(-2E)\boxtimes\O_{\P(L)}(1))\cong H^0(\O_{\P^4\times\P(L)}(2,1))\oplus
    H^0(\O_{\P^4\times\P(L)}(3,1)),
\]
and thus the tautological section which cuts out the universal hyperplane
section~$\TH\subset\Bl_P\P^5\times\P(L)$ induces a section cutting out~$\CX\subset\P^4\times\P(L)$.
An argument by Calabrese and Thomas~\cite[Lemma~4.3]{calabresethomas} shows that we have an
isomorphism
\begin{equation}
    \label{eq:pointiso}
    \TH\cong\Bl_{\CX}\left(\P^4\times\P(L)\right),
\end{equation}
where the projection~$\cpi\colon\TH\to\P^4\times\P(L)$ is given by
\[
    \cpi=(\phi\times\id_{\P(L)})\circ i_{\TH}.
\]
We will denote the exceptional divisor of this blowup by~$\CE$. Recall from the previous sections
that we also have an isomorphism
\[
    \TH\cong\Bl_{P\times\P(L)}\H_L.
\]
Looking at the defining equations one can compute explicitly that the exceptional divisor~$\E$ of
this blowup gets mapped to the~$(2,1)$ divisor containing~$\CX$ under the composition of the
isomorphism~\eqref{eq:pointiso} and the projection~$\cpi$. In fact~$\E$ is the proper transform of
this divisor and so we compute
\begin{align*}
    \O(\E)&\cong\cpi^*\O_{\P^4\times\P(L)}(2,1)(-\CE)\\
    &\cong i_{\TH}^*\left(\pi^*\O_{\P^5}(2)(-2E)\boxtimes\O_{\P(L)}(1)\right)(-\CE).
\end{align*}
Using the fact that~$\O(\E)\cong i_{\TH}^*(\O(E)\boxtimes\O_{\P(L)})$ we can rewrite this as
\[
    \O(\CE)\cong i_{\TH}^*\left(\pi^*\O_{\P^5}(2)(-3E)\boxtimes\O_{\P(L)}(1)\right).
\]
To summarise, we have the following commutative diagram:
\[
    \begin{tikzcd}
        E\times\P(L)\ar[r,"j\times\id"]&\Bl_P\P^5\times\P(L)&\TH\ar[l,hook',"i_{\TH}"]\ar[d,"\cong"]&\\
        \E\ar[r,hook]\ar[d]\ar[u,hook,"i_{\TH}"]&\Bl_{P\times\P(L)}\H_L\ar[d]\ar[r,"\cong"]\ar[u,hook,"i_{\TH}"]&\Bl_{\CX}(\P^4\times\P(L))\ar[d,"\cpi"]&\CE\ar[l,hook',"\cj"]\ar[d,"\cp"]\\
        P\times\P(L)\ar[r,hook]&\H_L&\P^4\times\P(L)&\CX\rlap{\ .}\ar[l,hook',"\ci"]
    \end{tikzcd}
\]

We are now ready to start imitating Kuznetsov's mutations~\cite{kuznetsovcubic} to prove the
proposition.  First note that after applying an overall twist by~$\pi^*\O_{\P^5}(3)(-2E)$ to the
semiorthogonal decomposition of~$\D(\TH)$ coming from~\eqref{eq:lef1stexample} we get:
\begin{multline*}
    \D(\TH)=\ang[\Big]{j_*\O_E\otimes\pi^*\O_{\P^5}(-2)(2E)\boxtimes\D(\P(L)),j_*\O_E\otimes\pi^*\O_{\P^5}(-1)(E)\boxtimes\D(\P(L)),\\
        \TC,\pi^*\O_{\P^5}\boxtimes\D(\P(L)),\pi^*\O_{\P^5}(1)\boxtimes\D(\P(L)),\pi^*\O_{\P^5}(2)\boxtimes\D(\P(L))}.
\end{multline*}
Note that in the case of~$\D(P)_{-4}$ we chose to write the structure sheaf of the point
as~$i^*\O_{\P^5}(-5)$, whereas in the case of~$\D(P)_{-3}$ we chose to write it
as~$i^*\O_{\P^5}(-4)$; the reasons for this will become clear later. Let us now show that~$\TC$ is
equivalent to~$\D(\CX)$. To do this we begin with the semiorthogonal decomposition of~$\TH$ coming
from the isomorphism~\eqref{eq:pointiso} and Orlov's theorem:
\begin{multline*}
    \D(\TH)=\ang[\Big]{\Phi(\D(\CX)),\pi^*\O_{\P^5}(-3)(3E)\boxtimes\D(\P(L)),\\
        \pi^*\O_{\P^5}(-2)(2E)\boxtimes\D(\P(L)),\dots,\pi^*\O_{\P^5}(1)(-E)\boxtimes\D(\P(L))}.
\end{multline*}
Here~$\Phi$ is the embedding given by~$\Phi(\dash)=\cj_*\cp^*(\dash)\otimes\O(\CE)$ and we chose to
decompose~$\D(\P^4)$ as
\[
    \D(\P^4)=\ang{\O_{\P^4}(-3),\O_{\P^4}(-2),\dots,\O_{\P^4}(1)}.
\]
In the above semiorthogonal decompositions of~$\D(\TH)$ and all that follow we will usually suppress
writing the implicit restriction~$i_{\TH}^*$ since it is always fully faithful on all the pieces
that appear and we implicitly keep applying the following lemma:

\begin{lem}
    \label{lem:mutationembedding}
    Let~$\A\subset\D(X)$ be an admissible full triangulated subcategory. If~$\F\in\A$ is an
    exceptional object, then for any~$\G\in\A$ we have isomorphisms
    \begin{gather*}
        \L_{i_{\A}(\F)}(i_{\A}(\G))\cong i_{\A}(\L_{\F}(\G)),\\
        \R_{i_{\A}(\F)}(i_{\A}(\G))\cong i_{\A}(\R_{\F}(\G)),
    \end{gather*}
    where~$i_{\A}\colon\A\hookrightarrow\T$ denotes the inclusion functor.
\end{lem}

\begin{proof}
    This is just a simple consequence of the fact that one can write mutation through an exceptional
    object explicitly. E.g.\ in the case of left mutation we compute
    \begin{align*}
        \L_{i_{\A}(\F)}(i_{\A}(\G))&=\Cone(i_{\A}(\RHom(\F,i_{\A}^!(i_{\A}(\G)))\otimes\F)\to
        i_{\A}(\G))\\
        &=i_{\A}(\Cone(\RHom(\F,\G)\otimes\F\to\G))\\
        &=i_{\A}(\L_{\F}(\G)).
    \end{align*}
    The case of right mutation is analogous.
\end{proof}

\subsubsection*{Step 1}
We mutate the first three pieces after~$\Phi(\D(\CX))$ all the way to the left. To do this we apply
Proposition~\ref{prop:lefgeneral} to obtain
\[
    \L_{\Phi(\D(\CX))}(\cpi^*\F)=\cpi^*\F(\CE).
\]
Using the previous expression for~$\O(\CE)$ we obtain the following semiorthogonal decomposition
after the mutation:
\begin{multline*}
    \D(\TH)=\ang[\Big]{\pi^*\O_{\P^5}(-1)\boxtimes\D(\P(L)),\pi^*\O_{\P^5}(-E)\boxtimes\D(\P(L)),\\
        \pi^*\O_{\P^5}(1)(-2E)\boxtimes\D(\P(L)),\Phi(\D(\CX)),\pi^*\O_{\P^5}\boxtimes\D(\P(L)),\pi^*\O_{\P^5}(1)(-E)\boxtimes\D(\P(L))}.
\end{multline*}

\subsubsection*{Step 2}
We mutate~$\Phi(\D(\CX))$ all the way to the right. All we need to remark at this point is
that~$\Phi$ is given by the Fourier--Mukai transform whose kernel is
\[
    (\cp\times\cj)_*\O_{\CE}(-1),
\]
which is of course supported on the fibre product~$\CX\times_{\P(L)}\TH$. We can write the result of
the mutation in terms of a modified embedding
\[
    \Phi'=\R_{\pi^*\O_{\P^5}(1)(-E)\boxtimes\D(\P(L))}\circ\R_{\pi^*\O_{\P^5}\boxtimes\D(\P(L))}\circ\Phi,
\]
and we just want to make sure that~$\Phi'$ can still be written as a Fourier--Mukai transform whose
kernel is supported on~$\CX\times_{\P(L)}\TH$. This fact is an immediate consequence of the
following two lemmas:

\begin{lem}
    \label{lem:mukai}
    Let~$X,Y,Z$ be smooth projective varieties equipped with regular maps to a smooth projective
    variety~$B$ and assume that either one of the projections $X\times_B Y\to Y$ or~$Y\times_B Z\to
    Y$ is flat and~$Y$ is proper over~$B$. Given $\F\in\mathbf{D}^{b}(X\times_B Y)$
    and~$\G\in\mathbf{D}^{b}(Y\times_B Z)$, if either one of the kernels is perfect, the relative
    convolution
    \[
        \pi_{X,Z *}^B(\pi_{X,Y}^{B,*}\F\otimes\pi_{Y,Z}^{B,*}\G)
    \]
    gives an isomorphism 
    \[
        \Phi_{Y\to Z}^{B,\G}\circ\Phi_{X\to Y}^{B,\F}\cong\Phi_{X\to Z}^{B,\F*\G},
    \]
    with~$\F*\G\in\D(X\times_B Z)$.
\end{lem}

\begin{proof}
    This is just a matter of looking at the diagram of the corresponding proposition in Huybrechts'
    book~\cite[Proposition~5.10]{huybrechts} and noting that with the flatness assumption we get an
    exact Cartesian square~\cite[Section~2.6]{kuznetsovhpd} and everything else generalises from
    products to fibre products. Observer that if one prefers to work with absolute integral
    functors, one just notices that
    \[
     \Phi_{X\to Y}^{B,\F}\cong\Phi_{X\to Y}^{j_{X,Y,*}\F}
    \]
    where~$j_{X,Y}\colon X\times_B Y\to X\times Y$.
\end{proof}

\begin{lem}
    \label{lem:fibreproductkernel}
    Let~$i_Y\colon Y\hookrightarrow X\times B$ be a closed immersion where all the varieties are
    smooth and projective, and assume that there is an exceptional object~$\F\in\D(X)$ such that the
    derived pullback functor~$i_Y^*$ is fully faithful on~$\ang{\F}\boxtimes\D(B)$. Then both left
    and right mutation through $i_Y^*(\F\boxtimes\D(B))$ in~$\D(Y)$ can be written as a
    Fourier--Mukai transform whose kernel is supported on~$Y\times_B Y$.
\end{lem}

\begin{proof}
    Note first that the left adjoint~$i_{Y,!}$ of the derived pullback~$i_Y^*$ is given by
    \[
        i_{Y,!}(\dash)=i_{Y,*}(\dash\otimes\Lambda^c\N[-c]),
    \]
    where~$\N$ is the normal bundle of $Y$ in $X\times B$ and~$c$ is the codimension. Now we can say
    that left and right mutation through $i_Y^*(\F\boxtimes\D(B))$ in~$\D(Y)$ are given by the
    Fourier--Mukai transforms with kernels
    \begin{gather*}
        \Cone(\K_{\L}\to\O_{\Delta_Y}),\\
        \Cone(\O_{\Delta_Y}\to\K_{\R})[-1],
    \end{gather*}
    respectively, where $\K_{\L}$ and~$\K_{\R}$ are the convolutions
    \begin{gather*}
        \K_{\L}=((\id_Y\times
        i_Y)_*\O_Y)*(\F^{\vee}\boxtimes\F\boxtimes\O_{\Delta_B})*((i_Y\times\id_Y)_*\O_Y),\\
        \K_{\R}=((\id_Y\times i_Y)_*\Lambda^c\N[-c])*((\F^{\vee}\otimes\omega_X[\dim
        X])\boxtimes\F\boxtimes\O_{\Delta_B})*((i_Y\times\id_Y)_*\O_Y).
    \end{gather*}
    Observe that the graph map $(\id_Y\times_B i_Y)\colon Y\to Y\times_B (X\times B)$ is a regular
    embedding, so by Lemma~\ref{lem:mukai} it is possible to consider the relative convolution. We
    can actually compute the convoluted kernels explicitly, getting:
    \begin{gather*}
        \K_{\L}=(i_Y\times_B
        i_Y)^*(\F^{\vee}\boxtimes\F\boxtimes\O_{\Delta_B}),\\
        \K_{\R}=(i_Y\times_B i_Y)^*((\F^{\vee}\otimes\omega_X[\dim
        X])\boxtimes\F\boxtimes\O_{\Delta_B})\otimes\pi_Y^*\Lambda^c\N[-c],
    \end{gather*}
    where we denoted by $\pi_Y\colon Y\times_B Y\to Y$ the projection onto the second factor.
\end{proof}

Looking at the explicit expression of~$\Cone(\O_{\Delta_Y}\to\K_{\R})[-1]$ we see that for~$\F$
locally free, the mutations through~$i^*_{\TH}(\F\boxtimes\D(\P(L))$ can be written as relative
Fourier--Mukai transforms whose kernels are actually perfect. We thus conclude that after mutating
we have the semiorthogonal decomposition
\begin{multline*}
    \D(\TH)=\ang[\Big]{\pi^*\O_{\P^5}(-1)\boxtimes\D(\P(L)),\pi^*\O_{\P^5}(-E)\boxtimes\D(\P(L)),\\
        \pi^*\O_{\P^5}(1)(-2E)\boxtimes\D(\P(L)),\pi^*\O_{\P^5}\boxtimes\D(\P(L)),\pi^*\O_{\P^5}(1)(-E)\boxtimes\D(\P(L)),\Phi'(\D(\CX))}.
\end{multline*}

\subsubsection*{Step 3}
We now transpose the third and fourth piece of the semiorthogonal decomposition since they are
completely orthogonal. Indeed we have
\begin{align*}
    \Ext^{\bullet}_{\D(\Bl_P\P^5)}(\pi^*\O_{\P^5}(1)(-2E),\pi^*\O_{\P^5})=0
\end{align*}
by Proposition~\ref{prop:lefgeneral} and thus from the K\"unneth formula it follows that our pieces
are completely orthogonal. The semiorthogonal decomposition after this step is now
\begin{multline*}
    \D(\TH)=\ang[\Big]{\pi^*\O_{\P^5}(-1)\boxtimes\D(\P(L)),\pi^*\O_{\P^5}(-E)\boxtimes\D(\P(L)),\\
        \pi^*\O_{\P^5}\boxtimes\D(\P(L)),\pi^*\O_{\P^5}(1)(-2E)\boxtimes\D(\P(L)),\pi^*\O_{\P^5}(1)(-E)\boxtimes\D(\P(L)),\Phi'(\D(\CX))}.
\end{multline*}

\subsubsection*{Step 4}
We now right mutate the second piece through the third one and the fourth one through the fifth one.
We will only explain the first case since the second one works in exactly the same way. For that we
will decompose~$\D(\P(L))$ in the usual way so that we're effectively mutating exceptional objects.
Again by Proposition~\ref{prop:lefgeneral} and the K\"unneth formula we have for~$0\leq k'\leq k\leq
l$:
\[
    \dim\Ext^p(\pi^*\O_{\P^5}(-E)\boxtimes\O_{\P(L)}(k),\pi^*\O_{\P^5}\boxtimes\O_{\P(L)}(k'))=\begin{cases}
        1&\text{if $p=0$ and $k=k'$,}\\
        0&\text{else.}
    \end{cases}
\]
The first implication of this is that we can just swap pieces until we are in the situation
of~$k=k'$. In that case then we have the distinguished triangle
\[
    \R_{\pi^*\O_{P^5}}(\pi^*\O_{\P^5}(-E))\to\pi^*\O_{\P^5}(-E)\to\pi^*\O_{\P^5}\to\cdots,
\]
where we omitted the~$(\dash)\boxtimes\O_{\P(L)}(k)$ in every term. But we also have the short exact
sequence
\[
    0\to\pi^*\O_{\P^5}(-E)\to\pi^*\O_{\P^5}\to j_*\O_E\to 0,
\]
and thus after shifting and comparing we obtain
\[
    \R_{\pi^*\O_{\P^5}\boxtimes\O_{\P(L)}(k)}(\pi^*\O_{\P^5}(-E)\boxtimes\O_{\P(L)}(k))\cong
    j_*\O_E[-1]\boxtimes\O_{\P(L)}(k).
\]
With the same argument for the second case we obtain
\begin{multline*}
    \R_{\pi^*\O_{\P^5}(1)(-E)\boxtimes\O_{\P(L)}(k)}(\pi^*\O_{\P^5}(1)(-2E)\boxtimes\O_{\P(L)}(k))\cong\\
    \cong j_*\O_E\otimes\pi^*\O_{\P^5}(1)(-E)[-1]\boxtimes\O_{\P(L)}(k).
\end{multline*}
Finally we note that cases where~$k'>k$ only arise once a piece has been properly mutated and then
we again have complete orthogonality and can finish with a series of swaps. The semiorthogonal
decomposition that we finally obtain is thus
\begin{multline*}
    \D(\TH)=\ang[\Big]{\pi^*\O_{\P^5}(-1)\boxtimes\D(\P(L)),\pi^*\O_{\P^5}\boxtimes\D(\P(L)),j_*\O_E\boxtimes\D(\P(L)),\\
        \pi^*\O_{\P^5}(1)(-E)\boxtimes\D(\P(L)),j_*\O_E\otimes\pi^*\O_{\P^5}(1)(-E)\boxtimes\D(\P(L)),\Phi'(\D(\CX))}.
\end{multline*}

\subsubsection*{Step 5}
We now left mutate the fourth piece through the third one. For that we again decompose~$\D(\P(L))$
as usual. By Proposition~\ref{prop:lefgeneral}, the K\"unneth formula and the fact that~$j$ is a
divisorial embedding we now have for~$0\leq k'\leq k\leq l$:
\[
    \dim\Ext^p(j_*\O_E\boxtimes\O_{\P(L)}(k),\pi^*\O_{\P^5}(1)(-E)\boxtimes\O_{\P(L)}(k'))=\begin{cases}
        1&\text{if $p=1$ and $k=k'$,}\\
        0&\text{else.}
    \end{cases}
\]
It follows again that we can just swap pieces until we are in the case of~$k=k'$ and we have the
distinguished triangle
\[
    j_*\O_E[-1]\to\pi^*\O_{\P^5}(1)(-E)\to\L_{j_*\O_E}(\pi^*\O_{\P^5}(1)(-E))\to\cdots,
\]
where we again omitted~$(\dash)\boxtimes\O_{\P(L)}(k)$ in every term. Since we are blowing up a
point we have the short exact sequence
\[
    0\to\pi^*\O_{\P^5}(1)(-E)\to\pi^*\O_{\P^5}(1)\to j_*\O_E\to 0.
\]
Thus after shifting and comparing we obtain
\[
    \L_{j_*\O_E\boxtimes\O_{\P(L)}(k)}(\pi^*\O_{\P^5}(1)(-E)\boxtimes\O_{\P(L)}(k))\cong\pi^*\O_{\P^5}(1)\boxtimes\O_{\P(L)}(k).
\]
Hence we now have the semiorthogonal decomposition
\begin{multline*}
    \D(\TH)=\ang[\Big]{\pi^*\O_{\P^5}(-1)\boxtimes\D(\P(L)),\pi^*\O_{\P^5}\boxtimes\D(\P(L)),\pi^*\O_{\P^5}(1)\boxtimes\D(\P(L)),\\
        j_*\O_E\boxtimes\D(\P(L)),j_*\O_E\otimes\pi^*\O_{\P^5}(1)(-E)\boxtimes\D(\P(L)),\Phi'(\D(\CX))}.
\end{multline*}

\subsubsection*{Step 6}
Finally we mutate the rightmost three pieces all the way to left. For this we note that the
canonical bundle of~$\TH$ is
\[
    \omega_{\TH}=i_{\TH}^*\left(\pi^*\O_{\P^5}(-3)(2E)\boxtimes\omega_{\P(L)}(1)\right),
\]
and thus after an additional twist by~$i_{\TH}^*\left(\pi^*\O_{\P^5}(1)\boxtimes\O_{\P(L)}\right)$
we obtain
\begin{multline*}
    \D(\TH)=\ang[\Big]{j_*\O_E\otimes\pi^*\O_{\P^5}(-2)(2E)\boxtimes\D(\P(L)),j_*\O_E\otimes\pi^*\O_{\P^5}(-1)(E)\boxtimes\D(\P(L)),\\
        \Phi''(\D(\CX)),\pi^*\O_{\P^5}\boxtimes\D(\P(L)),\pi^*\O_{\P^5}(1)\boxtimes\D(\P(L)),\pi^*\O_{\P^5}(2)\boxtimes\D(\P(L))}.
\end{multline*}
Note that~$\Phi''$ is just~$\Phi'$ composed with tensoring and thus still a Fourier--Mukai transform
whose kernel is supported on~$\CX\times_{\P(L)}\TH$; hence we are done.

\begin{rem}
    Note that by Theorem~\ref{thm:lefmult} we see that~$\D(\CX)$ is a categorical resolution of
    singularities of the noncommutative K3-fibration~$\C_L\to\P(L)$. This is basically a family
    version of Kuznetsov's result on singular cubic fourfolds~\cite{kuznetsovcubic}. By taking a
    generic pencil we also recover Calabrese and Thomas' example of derived equivalent Calabi--Yau
    threefolds~\cite{calabresethomas}.
\end{rem}

\subsection{Second example}
Consider now the linear system 
\[
    L=H^0(\I_{\P(W)}(3))\subset H^0(\O_{\P^5}(3))
\]
of cubic fourfolds containing a plane $\P(W)$. Looking at the blowup~$\Bl_{\P(W)}\P^5$, we can
identify~$L$ with the complete linear system~$H^0(\pi^*\O_{\P^5}(3)(-E))$, where as usual~$\pi$ denotes the
projection from the blow up. The line bundle~$\pi^*\O_{\P^5}(3)(-E)$ has no base
locus~\cite[Lemma~3.5]{calabresethomas} so we get a regular map
\[
    \begin{tikzcd}
        \Bl_{\P(W)}\P^5\ar[d,"\pi"]\ar[dr,"f"]&\\
        \P^5\ar[r,dashed]&\P(L^*)\rlap{\ .}
    \end{tikzcd}
\]
By Proposition~\ref{prop:lefgeneral} the following is a Lefschetz decomposition for
$\Bl_{\P(W)}\P^5$ with respect to the line bundle $\pi^*\O(3)(-E)$:
\[
\D(\Bl_{\P(W)}\P^5)=\ang*{\TA,\TA(1)},
\]
where we recall that
\[
    \TA\coloneqq\ang*{\pi^*\A\otimes\O(2E),\D(\P(W))_{-2}},
\]
$\A$ being the first component in the rectangular Lefschetz decomposition of $\D(\P^5)$ with respect
to $\O_{\P^5}(3)$. Write $\P^5=\P(W\oplus W')$.

\begin{prop}
    The HP dual of~$\Bl_{\P(W)}\P^5\to\P(L^*)$ with respect to the above Lefschetz decomposition is
    the noncommutative variety $(\P(W')\times\P(L),\C_0)$, where~$\C_0$ is the even Clifford algebra
    sheaf on~$\P(W')\times\P(L)$ corresponding to the quadric fibration~$\check{\pi}$:
    \[
        \begin{tikzcd}
            \TH\ar[dr,"\check{\pi}"']\ar[r,"i_{\TH}",hook]
            &\P\left(\underline{W}(1)\oplus\O_{\P(W')}\boxtimes\O_{\P(L)}(1)\right)\ar[d,"\phi\times
            \id_{\P(L)}"]\\
            &\P(W')\times\P(L)\rlap{\ .}
        \end{tikzcd}
    \]
\end{prop}

The claim indeed follows as a direct generalisation of \cite[Section~3]{calabresethomas},
\cite[Section~4]{kuznetsovcubic}. Let's start explaining why $\check{\pi}$ is a quadric fibration.
The linear projection $\P^5\dashrightarrow\P(W')$ induces a regular map
\[
    \begin{tikzcd}
        \Bl_{\P(W)}\P^5\ar[r,"\phi"]& \P(W')\rlap{\ .}
    \end{tikzcd}
\]
As pointed out in \cite{calabresethomas}, $\phi$ is the $\P^3$-bundle given by the projective
completion $\P\left(W\otimes\O_{\P(W')}(1)\oplus\O_{\P(W')}\right)$ of $W\otimes\O_{\P(W')}(1)$ over
$\P(W')$. Computing the Grothendieck line bundle one gets
$\O_{\phi}(1)\cong\O_{\widetilde{\P^5}}(E)$; this computation together with the canonical relation
\begin{equation}
    \label{hande}
    \pi^*\O_{\P^5}(1)(-E)=\phi^*\O_{\P(W')}(1)
\end{equation}
gives
\[
    \pi^*\O_{\P^5}(3)(-E)=\O_{\phi}(2)\otimes\phi^*\O_{\P(W')}(3).
\]
This means that the universal family $\TH$ that we know is cut out by the tautological section of
$\pi^*\O_{\P^5}(3)(-E)\boxtimes\O_{\P(L)}(1)$ is an element of the linear system
\[
    \abs*{\O_{\phi}(2)\otimes\phi^*\O_{\P(W')}(3)\boxtimes\O_{\P(L)}(1)},
\]
namely, it is cut out by a section of
$S^2(\underline{W}^*\oplus\O_{P(W')}(1))\otimes\O_{\P(W')\times\P(L)}(1,1)$ (we refer again to
\cite[Section~3]{calabresethomas} for the details of the computation). By
\cite[Theorem~4.2]{kuznetsovquadric} there exists a semiorthogonal decomposition
\begin{multline}
    \label{lef1}
    \D(\TH)=\ang[\big]{\Phi\left(\D(\P(W')\times\P(L)\right),\C_0),\check{\pi}^*\D(\P(W')\times\P(L))\otimes\O_{\check{\pi}}(1),\\
        \check{\pi}^*\D(\P(W')\times\P(L))\otimes\O_{\check{\pi}}(2)}
\end{multline}
where we are denoting with $\O_{\check{\pi}}(1)$ the restriction to $\TH$ of the tautological line
bundle on $\P\left(\underline{W}(1)\oplus\O_{\P(W')}\boxtimes\O_{\P(L)}(1)\right)$. Note that 
\[
    \O_{\check{\pi}}(1)\cong i_{\TH}^*\left(\O_{\phi}(1)\boxtimes\O_{\P(L)}(1)\right)\cong
    i_{\TH}^*\left(\O_{\widetilde{\P^5}}(E)\boxtimes\O_{\P(L)}(1)\right).
\]
After an overall twist by~$\O(-E,0)$ we can rewrite \eqref{lef1} in the following way:
\begin{multline}
    \label{lef2}
    \D(\TH)=\ang[\big]{\Phi'\left(\D(\P(W')\times\P(L),\C_0)\right),\phi^*\D(\P(W')\boxtimes\D(\P(L)),\\
        \phi^*\D(\P(W')(H)\boxtimes\D(\P(L))},
\end{multline}
where we denote by $H$ the twisting by $\pi^*\O_{\P^5}(1)$ and by $\Phi'$ the faithful embedding
given by $\Phi'(\dash)=\Phi(\dash)\otimes\O(-E,0)$. In \eqref{lef2} and in all that follows, we drop
the $i_{\TH}^*$ to simplify the notation, as it is fully faithful anyway. Before going into
mutations, we rewrite \eqref{lef2} in a more explicit way. Replace the first instance of
$\D(\P(W'))$ with the exceptional collection
$\left(\O_{\P(W')}(-1),\O_{\P(W')},\O_{\P(W')}(1)\right)$ and the second one with the exceptional
collection $\left(\O_{\P(W')},\O_{\P(W')}(1),\O_{\P(W')}(2)\right)$ and denote by $\O(h)$ the
twisting by $\phi^*\O_{\P(W')}(1)$. We get
\begin{multline}
    \label{lef3}
    \D(\TH)=\ang[\big]{\Phi'\left(\D(\P(W')\times\P(L),\C_0)\right),\O(-h)\boxtimes\D(\P(L)),\dots\\
        \dots\O(2h+H)\boxtimes\D(\P(L))}.
\end{multline}
On the other hand with~$\TH$ being the universal hyperplane section of~$\Bl_{\P(W)} \P^5$ with
respect to~$\pi^*\O_{\P^5}(3)(-E)$, it has a semiorthogonal decomposition of the form
\[
    \D(\TH)=\ang*{\TC,\TA(1)\boxtimes\D(\P(L))}
\]
Replace $\A$ with the exceptional collection $\ang{\O_{\P^5}(3),\O_{\P^5}(4), \O_{\P^5}(5)}$ and
$\D(\P(W))$ with the exceptional collection $\ang{\O_{\P(W)}(-3),\O_{\P(W)}(-2),\O_{\P(W)}(-1)}$.
The reason for these choices will be clear in a moment; writing this down explicitly we get:
\begin{multline*}
    \TA=\ang[\big]{\pi^*\O_{\P^5}(-3)(2E),\pi^*\O_{\P^5}(-2)(2E),\pi^*\O_{\P^5}(-1)(2E),j_*p^*\O_{\P^2}(-3)(2E),\\
        j_* p^*\O_{\P^2}(-2)(2E),j_*p^*\O_{\P^2}(-1)(2E)}.
\end{multline*}
We then get the following semiorthogonal decomposition for $\D(\TH)$:
\begin{multline}
    \label{lefb1}
    \D(\TH)=\ang[\big]{\TC',\pi^*\O_{\P^5}\boxtimes\D(P(L),\pi^*\O_{\P^5}(1)\boxtimes\D(\P(L)),\dots\\
        \dots j_*\O_E\otimes\pi^*\O_{\P^5}(2))\boxtimes\D(\P(L))}.
\end{multline}
where we denote by $\TC'$ the category $\TC\otimes\O(-E,0)$. Note that in \eqref{lefb1}, to get the
last three pieces of the semiorthogonal decomposition we used that
\begin{equation*}
    j_*(p^*\O_{\P(W)}(t)\otimes\O_E)\cong j_*(p^*i^*\O_{\P^5}(t)\otimes\O_E)\cong
    j_*(j^*\pi^*\O_{\P^5}(t)\otimes\O_E)
\end{equation*}
and the projection formula. Summing up, we get
\begin{multline}
    \label{lefhyp}
    \D(\TH)=\ang[\big]{\TC',\O\boxtimes\D(\P(L)),\O(H)\boxtimes\D(\P(L)),\\
        \O(2H)\boxtimes\D(\P(L)),j_*\O_E\boxtimes\D(\P(L)),j_*\O_E(H)\boxtimes\D(\P(L)),\\
        j_*\O_E(2H)\boxtimes\D(\P(L))}
\end{multline}

We are now ready to follow Kuznetsov's recipe for mutation~\cite[Section~4]{kuznetsovcubic} to
show that the interesting part $\TC'$ is equivalent to $\D(\P(W')\times\P(L),\C_0)$.

\subsubsection*{Step~1}
We start from \eqref{lef3} right mutating $\Phi'\left(\D(\P(W')\times\P(L),\C_0)\right)$ through
$\O(-h)\boxtimes\D(\P(L))$. This leads to the following decomposition:
\begin{multline}
    \label{lef4}
    \D(\TH)=\ang[\big]{\O(-h)\boxtimes\D(\P(L)),\Phi''\left(\D(\P(W')\times\P(L),\C_0)\right),\\
        \O\boxtimes\D(\P(L)),\dots\O(2h+H)\boxtimes\D(\P(L))},
\end{multline}
where $\Phi''=\R_{\O(-h)\boxtimes\D(\P(L))}\circ\Phi'$.

\subsubsection*{Step~2}
We now want to mutate $\O(-h)\boxtimes\D(\P(L))$ through the left orthogonal subcategory $^\perp
\ang*{\O(-h)\boxtimes\D(\P(L))}$. The anticanonical bundle of~$\TH$ is
\[
    \omega_{\TH}^{\vee}=i_{\TH}^*\left(\O(2H+h)\boxtimes\omega_{\P(L)}^{\vee}(-1)\right),
\]
and thus we obtain
\begin{multline}
    \label{lef5}
    \D(\TH)=\ang[\big]{\Phi''\left(\D(\P(W')\times\P(L),\C_0)\right),\\
        \O\boxtimes\D(\P(L)),\dots\O(2h+H)\boxtimes\D(\P(L)),\O(2H)\boxtimes\D(\P(L))}.
\end{multline}

\subsubsection*{Step~3}
Transpose the pair $\left(\O(2h+H)\boxtimes\D(\P(L)),\O(2H)\boxtimes\D(\P(L))\right)$. Using the
K\"unneth formula we reduce to showing that
\begin{align*}
    \Ext^{\bullet}_{\widetilde{\P^5}}\left(\O(2h+H),\O(2H)\right)\boxtimes\Ext^{\bullet}_{\P(L)}(F,G)
\end{align*}
vanishes for every $F,G\in\D(\P(L))$. Indeed we have
\begin{align*}
    \Ext^{\bullet}(\O(2h+H),\O(2H))&=H^{\bullet}(\Bl_{\P(W)}\P^5,\pi^*\O_{\P^5}(-1)(2E))\\
    &=H^{\bullet}(\P_{\P(W')}(\underline{W}(1)\oplus\O_{\P(W')}),\pi^*\O_{\P(W)}(-1)\otimes\O_{\phi}(1))\\
    &=H^{\bullet}(\P(W'),\O_{\P(W')}(-1)\otimes(\underline{W}(-1)\oplus\O_{\P(W')}))\\
    &=0.
\end{align*}
After transposing we have
\begin{multline}
    \label{lef6}
    \D(\TH)=\ang[\big]{\Phi''\left(\D(\P(W')\times\P(L),\C_0)\right),\\
        \O\boxtimes\D(\P(L)),\dots\O(2H)\boxtimes\D(\P(L)),\O(2h+H)\boxtimes\D(\P(L))}.
\end{multline}

\subsubsection*{Step~4}
We left mutate $\O(2h+H)\boxtimes\D(\P(L))$ trough the right orthogonal subcategory
$\ang*{\O(2h+H)\boxtimes\D(\P(L))}^{\perp}$. This just means twisting by the canonical bundle which
we computed in Step~2:
\begin{multline}
    \label{lef7}
    \D(\TH)=\ang[\big]{\O(h-H)\boxtimes\D(\P(L)),\Phi''\left(\D(\P(W')\times\P(L),\C_0)\right),\\
        \O\boxtimes\D(\P(L)),\dots\O(2H)\boxtimes\D(\P(L))}.
\end{multline}

\subsubsection*{Step~5}
Left mutation of $\Phi''\left(\D(\P(W')\times\P(L),\C_0)\right)$ through $\O(h-H)\boxtimes\D(\P(L))$
gives us the following decomposition:
\begin{multline}
    \label{lef8}
    \D(\TH)=\ang[\big]{\Phi'''\left(\D(\P(W')\times\P(L),\C_0\right),\O(h-H)\boxtimes\D(\P(L)),\\
        \O\boxtimes\D(\P(L)),\dots\O(2H)\boxtimes\D(\P(L))},
\end{multline}
where $\Phi'''=\L_{\O(h-H)\boxtimes\D(\P(L))}\circ \Phi''$. Let's remark at this point that the
embedding $\Phi'''$ is a kernel functor whose kernel is supported on
$(\P(W')\times\P(L),\C_0)\times_{\P(L)}\TH$. This can be seen as in the first example, except now
one has to pay attention to the fact that we are dealing with a noncommutative variety. Using
technical results of Kuznetsov~\cite[Appendix~D]{kuznetsovhyperplane} one can check that everything
goes through as expected.

\subsubsection*{Step~6}
We want to simultaneously mutate $\O(h-H)\boxtimes\D(\P(L))$ through
$\O_{\tilde{P^5}}\boxtimes\D(\P(L))$, $\O(h)\boxtimes\D(\P(L))$ trough $\O(H)\boxtimes\D(\P(L))$ and
$\O(h+H)\boxtimes\D(\P(L))$ trough $\O(2H)\boxtimes\D(\P(L))$. As we already saw in the previous
example, using Lemma~\ref{lem:mutationembedding} these mutations can be immediately computed. Write 
\[
    \D(\P(L))=\ang{\O_{\P(L)},\O_{\P(L)}(1),\dots,\O_{\P(L)}(l-1)}
\]
and perform the mutations objectwise. Using the K\"unneth formula we get for any $0\leq k'\leq k\leq
l$:
\[
    \dim\Ext^p(\O_{\tilde{\P^5}}(-E)\boxtimes\O_{\P(L)}(k),\O_{\tilde{\P^5}}\boxtimes\O_{\P(L)}(k'))=\begin{cases}
        1&\text{if $p=0$ and $k=k'$,}\\
        0&\text{else.}
    \end{cases}
\]
The computation is the same as in Step 3. The first implication of this is that we can just swap
pieces until we are in the situation of~$k=k'$. In that case then we have the distinguished triangle
\[
    \R_{\O_{\tilde{\P^5}}\boxtimes\O_{\P(L)}(k)}(\O_{\tilde{\P^5}}(-E)\boxtimes\O_{\P(L)}(k))\to\O_{\tilde{\P^5}}(-E)\boxtimes\O_{\P(L)}(k)\to\O_{\tilde{\P^5}}
\]
where the second map is the obvious one. It then follows that
\[
    \R_{\O_{\tilde{\P^5}}\boxtimes\O_{\P(L)}(k)}(\O_{\tilde{\P^5}}(-E)\boxtimes\O_{\P(L)}(k))=j_*\O_E[-1]\boxtimes\O_{\P(L)}(k),
\]
and similarly one also shows that
\begin{gather*}
    \R_{\O_{\tilde{\P^5}}(H)\boxtimes\O_{\P(L)}(k)}(\O_{\tilde{\P^5}}(h)\boxtimes\O_{\P(L)}(k))=j_*\O_E(H)[-1]\boxtimes\O_{\P(L)}(k),\\
    \R_{\O_{\tilde{\P^5}}(2h)\boxtimes\O_{\P(L)}(k)}(\O_{\tilde{\P^5}}(h+H)\boxtimes\O_{\P(L)}(k))=j_*\O_E(2H)[-1]\boxtimes\O_{\P(L)}(k).
\end{gather*}
Finally, for any $k'>k$ we see that
$\R_{\O_{\tilde{\P^5}}\boxtimes\O_{\P(L)}(k)}(\O_{\tilde{\P^5}}(-E)\boxtimes\O_{\P(L)}(k))$ and
$\O_{\tilde{\P^5}}\boxtimes\O_{\P(L)}(k')$ are completely orthogonal and so we get that
\[
    \R_{\O_{\tilde{P^5}}\boxtimes\D(\P(L))}(\O(h-H)\boxtimes\D(\P(L)))=\R_{\O_{\tilde{\P^5}}}(\O_{\tilde{\P^5}}(-E))\boxtimes\D(\P(L)),
\]
and similarly for the other two mutations.  After this step we have the following semiorthogonal
decomposition:
\begin{multline}
    \D(\TH)=\ang[\big]{\Phi'''\left(\D(\P(W')\times\P(L),\C_0\right),\O\boxtimes\D(\P(L)),j_*\O_E\boxtimes\D(\P(L)),\\
        \O(H)\boxtimes\D(\P(L)),j_*\O_E(H)\boxtimes\D(\P(L)),\O(2H)\boxtimes\D(\P(L))\\
        j_*\O_E(2H)\boxtimes\D(\P(L))}.
\end{multline}

\subsubsection*{Step~7}
Finally we want to transpose $j_*\O_E(H)\boxtimes\D(\P(L))$ to the right of
$\O(2H)\boxtimes\D(\P(L))$ and $j_*\O_E\boxtimes\D(\P(L))$ to the right of
$\ang[\big]{\O(H)\boxtimes\D(\P(L)), \O(2H)\boxtimes\D(\P(L))}$. We have to show complete
orthogonality; this follows easily using the same argument of
Kuznetsov~\cite[Lemma~4.6]{kuznetsovcubic}. Indeed
\begin{multline}
    \Ext^{\bullet}_{\TH}\left(\tilde{j}_*\tilde{p}^* F,\tilde{\pi}^*
        G\right)=\Ext^{\bullet}_{\TH}\left(\tilde{p}^* F,\tilde{j}^*\tilde{\pi}^*
        G\otimes\O_{\tilde{E}}(-1)[-1]\right)=\\
    \Ext^{\bullet}_{\TH}\left(\tilde{p}^*
        F,\tilde{p}^*\tilde{i}^*G\otimes\O_{\tilde{E}}(-1)[-1]\right)=\Ext^{\bullet}_{\TH}\left(F,\tilde{i}^*G\otimes\tilde{p_*}\O_{\tilde{E}}(-1)[-1]\right)=0
\end{multline}
where we used the notation of diagram~\eqref{eq:diag1} and the adjunction formula for $j_*$. So
finally we get:
\begin{multline}
    \D(\TH)=\ang[\big]{\Phi'''\left(\D(\P(W')\times\P(L),\C_0\right),\O\boxtimes\D(\P(L)),\O(H)\boxtimes\D(\P(L)),\\
        \O(2H)\boxtimes\D(\P(L)),j_*\O_E\boxtimes\D(\P(L)),j_*\O_E(H)\boxtimes\D(\P(L))\\
        j_*\O_E(2H)\boxtimes\D(\P(L))}
\end{multline}
Comparing with \eqref{lefhyp}, we conclude that $\TC'\cong
\Phi'''\left(\D(\P(W')\times\P(L),\C_0\right)$.

\begin{rem}
    Note again by Theorem~\ref{thm:lefmult} that this is just a family version of Kuznetsov's result
    on cubic fourfolds containing a plane~\cite{kuznetsovcubic}. Again by taking a generic pencil we
    also recover the corresponding result of Calabrese and Thomas~\cite{calabresethomas}.
\end{rem}

\bibliographystyle{amsalpha}
\bibliography{bibliography.bib}
\end{document}